\documentclass[11pt,a4paper]{article}
\usepackage{bbm}

\usepackage[leqno]{amsmath}
\usepackage{amsfonts}
\usepackage{graphicx}
\usepackage{amsmath}
\usepackage{amssymb}
\usepackage{latexsym}
\usepackage{amsmath, amsfonts,amssymb, amsthm, euscript,makeidx,color,mathrsfs}
\usepackage{enumerate}

 \usepackage{subfigure}

\usepackage[colorlinks,linkcolor=blue,anchorcolor=green,citecolor=red]{hyperref}

\def\sqr#1#2{{\vcenter{\vbox{\hrule height.#2pt
              \hbox{\vrule width.#2pt height#1pt \kern#1pt \vrule width.#2pt}
              \hrule height.#2pt}}}}
\def\bal{\begin{aligned}}
\def\eal{\end{aligned}}

\def\5n{\negthinspace \negthinspace \negthinspace \negthinspace \negthinspace }
\def\4n{\negthinspace \negthinspace \negthinspace \negthinspace }
\def\3n{\negthinspace \negthinspace \negthinspace }
\def\2n{\negthinspace \negthinspace }
\def\1n{\negthinspace }

\def\dbR{\mathbb{R}}

\def\={\buildrel \triangle \over =}

\def\ds{\displaystyle}

\def\e{\varepsilon}

\def\l{\lambda}

\def\O{\Omega}
\DeclareMathOperator{\df}{d\!}
\def\no{\noindent}

\def\ms{\medskip}

\def\cd{\cdot}

\def\({\Big (}
\def\){\Big )}
\def\[{\Big[}
\def\]{\Big]}

\def\bde{\begin{definition}\label}
\def\ede{\end{definition}}
\def\be{\begin{equation}}
\def\bel{\begin{equation}\label}
\def\ee{\end{equation}}
\def\bt{\begin{theorem}\label}
\def\et{\end{theorem}}
\def\bc{\begin{corollary}\label}
\def\ec{\end{corollary}}
\def\bl{\begin{lemma}\label}
\def\el{\end{lemma}}
\def\bp{\begin{proposition}\label}
\def\ep{\end{proposition}}
\def\bas{\begin{assumption}\label}
\def\eas{\end{assumption}}
\def\br{\begin{remark}\label}
\def\er{\end{remark}}
\def\bex{\begin{example}\label}
\def\ex{\end{example}}
\def\ba{\begin{array}}
\def\ea{\end{array}}
\def\ed{\end{document}}

\def\square#1{\vbox{\hrule\hbox{\vrule height#1%
     \kern#1\vrule}\hrule}}
\def\rectangle#1#2{\vbox{\hrule\hbox{\vrule height#1%
     \kern#2\vrule}\hrule}}

\font\tenbb=msbm10 \font\sevenbb=msbm7 \font\fivebb=msbm5

\newfam\bbfam
\scriptscriptfont\bbfam=\fivebb \textfont\bbfam=\tenbb
\scriptfont\bbfam=\sevenbb

\newtheorem{theorem}{\hskip 1.3em Theorem}[section]
\newtheorem{definition}[theorem]{\hskip 1.3em Definition}
\newtheorem{proposition}[theorem]{\hskip 1.3em Proposition}
\newtheorem{corollary}[theorem]{\hskip 1.3em Corollary}
\newtheorem{lemma}[theorem]{\hskip 1.3em Lemma}
\newtheorem{remark}[theorem]{\hskip 1.3em Remark}
\newtheorem{example}[theorem]{\hskip 1.3em Example}

\newtheorem{assumption}[theorem]{\hskip 1.3em Assumption}

\makeatletter
   
   \@addtoreset{equation}{section}
\makeatother

\begin{document}

\title{\bf Time-varying bang-bang property of minimal controls for approximately null-controllable heat equations\thanks{This work is supported in part by
the National Natural Science Foundation of China (11526167, 11371375), the Fundamental Research Funds for the Central
Universities (SWU113038, XDJK2014C076),  the Natural Science
Foundation of CQCSTC (2015jcyjA00017).}}

\author{Ning Chen\footnote{School of Information Science and Engineering,
    Central South University, Changsha 410075, P.R. China (\tt{ningchen@csu.edu.cn})},~~
    Yanqing Wang\footnote{Corresponding author, School of Mathematics and Statistics, Southwest University, Chongqing 400715, P.R. China (\tt yqwang@amss.ac.cn)},~~ and~~Dong-Hui Yang\footnote{School of Mathematics and Statistics, and School of Information Science and Engineering, Central South University, Changsha 410075, P.R. China (\tt{donghyang@139.com})}}

\maketitle

\begin{abstract}
In this paper, optimal time control problems and optimal target control problems are studied for the approximately null-controllable heat equations. 
Compared with the existed results on these problems, the boundary of control variables are not  constants but time varying functions. The time-varying bang-bang property for optimal time control problem, and an equivalence theorem for optimal control problem and optimal target problem are obtained.
\end{abstract}

\ms

\no\bf Keywords: \rm Heat equation, bang-bang property, optimal time control problem, optimal target control problem

\ms

\no\bf AMS subject classification: \rm 35K05, 49J20

\maketitle

\section{Introduction}

Let $T$ be a positive number (can be taken $\infty$) and $\O$ be an open bounded 
domain with smooth 
boundary in $\dbR^N,\, N\geq 1$. Let $\omega$ be an open set of $\Omega$. 
Consider the following controlled system:
\begin{equation}\label{6.29.1}
\begin{cases}
\partial_t y(x,t)-\Delta y(x,t)=\chi_\omega\chi_{(\tau,T)}u(x,t), &\mbox{ in } \Omega\times (0,T),\\
y(x,t)=0, &\mbox{ on }\partial \Omega\times (0,T),\\
y(x,0)=y_0(x), &\mbox{ in } \Omega.
\end{cases}
\end{equation}
Here $y_0\in L^2(\Omega)$ is a given initial data, $u\in L^\infty(0,T; L^2(\Omega))$, $\tau\in[0,T)$, $\chi_\omega$ and $\chi_{(\tau, T)}$ stand for the  characteristic functions of $\omega$ and $(\tau, T)$, 
respectively. We denote the solution to \eqref{6.29.1} by $y(\cdot; \chi_{(\tau, T)}u, y_0)$ with initial data $y_0$ and control $u$.

In this paper, denote by $L^\infty_+(0,T)$ the subset of $L^\infty(0,T)$, whose element is almost surely positive. Denote by $\|\cdot\|$ and $\langle\cdot, \cdot\rangle$ the norm and inner product of $L^2(\Omega)$, respectively, and $B(0,r)$/$\bar B(0,r)$ the open/closed ball of $L^2(\Omega)$ with center 0 and radius $r>0$.

The null and approximate controllability of system \eqref{6.29.1} has been studied in many works (see, e.g. \cite{FPZ,FCZ}).
Especially, for each $\e>0$, since the system \eqref{6.29.1} is energy decaying, taking $u=0$, we have $\|y(t; 0, y_0)\|\leq \e$, when $t$ is large enough. By this, we can easily see that the system \eqref{6.29.1} is approximately null-controllable for large $T$. The reader can also refer to \cite{Apraiz-Escauriaza-Wang-Zhang14,Guo-Xu-Yang16, Guo-Yang15,Wang08, Wang-Xu13} for more discussions on controlled heat equations.

Three kinds of optimal control problems: the optimal time, target and norm control problems are important and interesting branches of optimization. For the deterministic systems, the reader can refer to \cite{Fattorini11} to obtain the recent results and open problems. The reader can also refer to \cite{Fattorini99, KW, LY,Lv10, MS, Wang08} for the optimal time control problems. For the stochastic ones, the  optimal norm control problems were considered in \cite{Wang-Yang-Yong-Yu16,Wang-Zhang15,Yong-Zhou99}  for controlled stochastic ordinary differential equations, and in \cite{Yang-Zhong16} for controlled stochastic heat equations. The reader can also refer to \cite{GL, Wang-Xu13, WZ} for  the work on equivalence relation between these three optimal control problems.

\ms

For a given function $M(\cd)\in L^\infty_+(0,T)$, we can define the admissible control set of controlled system \eqref{6.29.1}:
\begin{equation*}
\mathcal{U}_{M}\equiv\left\{ v\in L^\infty(0,T; L^2(\Omega))\bigm| \|v(t)\|\leq M(t) \mbox{ for a.e. } t\in (\tau, T)\right\},
\end{equation*}
and can denote the reachable set of system \eqref{6.29.1} with $u\in\mathcal{U}_{M}$ by
\begin{equation*}
\mathcal{R}(y_0,\tau,T)=\left\{ y(T; \chi_{(\tau, T)}u, y_0)\bigm| u\in \mathcal{U}_{M}\right\}
\end{equation*}
and
\begin{equation*}
\mathcal{R}(y_0,T)=\bigcup_{\tau\in (0,T)}\mathcal{R}(y_0,\tau,T).
\end{equation*}
By above discussion, without loss of generality,  we can assume that $\mathcal{R}(y_0,T)\cap \bar B(0,\e)\neq \emptyset$ for $\e>0$ and $T>0$. 
We need note that $y(T; \chi_{(\tau, T)}0,y_0)$ may not be in $\bar B(0,\e)$.

\ms

Consider the following optimal time control problem
\begin{equation}\label{1.14.1}
\tau(\e)=\sup_{\tau\in [0,T]}\left\{\tau\bigm| y(T; \chi_{(\tau, T)}u, y_0)\in \bar B(0,\e), u\in \mathcal{U}_{M}\right\}.
\end{equation}
If the optimal time control problem \eqref{1.14.1} is solvable,
i.e., there exist at least one $u^*\in\mathcal{U}_{M}$ such that $y(T; \chi_{(\tau(\e), T)}u^*, y_0)\in \bar B(0,\e)$, we call that $u^*$  an {\em optimal time control}. By by choosing the minimal sequence and applying the classical variational method, we can prove that the  optimal time control problem \eqref{1.14.1} has a solution $u^*$ (see Lemma \ref{1.24.8}).
What we are interested in is the following problem: 
\begin{equation*}\label{b-b}
\mbox{{\em Is the optimal time control $u^*$ satisfying the time-varying bang-bang property?}}
\end{equation*}

The (time-invariant) bang-bang property is a classical problem in control theory. There are many works on this topic (see, e.g.  \cite{MS, S,Wang08}).  In \cite{Wang08}, the author obtained the following the bang-bang property of a null-controlled heat equation. For a given positive constant $M_0$, define
\begin{equation*}
\mathcal{U}_{M_0}=\left\{ v\in L^\infty(0,T; L^2(\Omega))\bigm| \|v(t)\|\leq M_0 \mbox{ for a.e. } t\in (\tau, T)\right\}.
\end{equation*}
 Then, if $u^*\in \mathcal{U}_{M_0}$ is an optima time control respect to \eqref{1.14.1}, then the following time-invariant bang-bang property holds:
\begin{equation*}
\|u^*(t)\|=M_0 \mbox{ for a.e. } t\in (\tau(0),T).
\end{equation*}
\cite{WZhang} owns many interesting results, but the bang-bang property is also depend on the positive constant $M_0$. This work is inspired by \cite{Wang08, WZhang}. 
In this work, we study the time-varying bang-bang property of an approximately null-controllable heat equation. Compared with the problem studied in \cite{Wang08}, the boundary is not a constant $M_0$ but a function $M(\cd)$. Hence, the method used in \cite{Wang08} is not workable any more. Until now, to our best knowledge, there does not exist any work on this kind time-varying bang-bang property. The following is our first main result.


\begin{theorem}\label{1.14.2}
Suppose that $M(\cd)\in L^\infty_+(0,T)$, $\e>0$, and $\mathcal{R}(y_0,T)\cap\bar B(0,\e)\neq \emptyset$. Then there exists a unique optimal time control $u^*\in\mathcal{U}_{M}$ such that the 
optimal time control problem \eqref{1.14.1} is solvable. Moreover, the optimal control $u^*$ satisfies the following time-varying bang-bang property:
\begin{equation}\label{1.14.3}
\|u^*(t)\|=M(t) \mbox{ for a.e. } t\in (\tau(\e), T).
\end{equation}

\end{theorem}

\ms

Now, we consider the following optimal target control problem:
\begin{equation}\label{1.24.111}
\e(\tau)=\inf\left\{\|y(T; \chi_{(\tau, T)}u, y_0)\|\bigm| u\in \mathcal{U}_{M}\right\}.
\end{equation}
Define 
\begin{equation*}
\e_T= \|y(T; \chi_{(0,T)}0, y_0)\|.
\end{equation*}
It is obviously that $0\leq\e(\tau)\leq \e_T$ and $0\leq \tau(\e)\leq T$. As an application of the time-varying bang-bang property, we shall give our second main result: a kind of equivalence related to $\e(\tau)$ and $\tau(\e)$.

\begin{theorem}\label{1.14.4}
Let $M(\cd)\in L^\infty_+(0,T)$. Then the map $\tau\mapsto \e(\tau)$ is strictly monotonically increasing and continuous from $[0,T)$ onto $[\e(0), \e_T)$. Furthermore, it holds that 
\begin{equation}\label{1.14.5}
\e=\e(\tau(\e)), \e\in [\e(0), \e_T),\quad \mbox{ and }\quad \tau=\tau(\e(\tau)), \tau\in [0,T).
\end{equation}
Consequently, the maps $\tau\mapsto \e(\tau)$ and $\e\mapsto \tau(\e)$ are inverse of each other.
\end{theorem}

When $M(\cd)\equiv M_0$ ($M_0>0$ is a constant), a kind of equivalence theorem of optimal time and target control problems has been discussed in \cite{Wang-Xu13}. In our work, for the time variant function $M(\cdot)$, we can also obtain that equivalence result. 
\ms

We organize this paper as follows. In Section 2, we prove the time-varying bang-bang property (Theorem \ref{1.14.2}). In Section 3, we prove the equivalence theorem of optimal time and target control problems (Theorem \ref{1.14.4}).

\section{Proof of Theorem \ref{1.14.2}}

The following lemma is crucial in the proof of Theorem \ref{1.14.2}.

\begin{lemma}\label{1.24.8}
Under the assumption of Theorem \ref{1.14.2}, let $\tau(\e)$ be defined as \eqref{1.14.1}. Then there exists an optimal control $u^*\in\mathcal{U}_{M}$, such that the optimal time problem \eqref{1.14.1} is solvable, i.e.,
\begin{equation*}
y(T; \chi_{(\tau(\e),T)}u^*, y_0)\in \bar B(0,\e).
\end{equation*} 
\end{lemma}

\begin{proof}
Since $\mathcal{R}(y_0,T)\cap\bar B(0,\e)\neq \emptyset$, there exist $\tau_0\in (0,T)$ and $u\in \mathcal{U}_M$, such that $y(T; \chi_{(\tau_0,T)}u, y_0)\in \bar B(0,\e)$. 
Let $\{\tau_n\}_{n=1}^\infty$ be a monotonically increasing sequence such that $\tau_n\rightarrow \tau(\e)$. Then, for each $n\in\mathbb{N}$, there exists $u_n\in \mathcal{U}_{M}$ such that 
\begin{equation*}
y(T; \chi_{(\tau_n, T)}u_n, y_0)\in \bar B(0,\e).
\end{equation*}
Set 
\begin{equation*}
\tilde u_n(t)=
\begin{cases}
0, & \mbox{ if } t\in (0, \tau_n],\\
u_n(t), &\mbox{ if } t\in (\tau_n, T).
\end{cases}
\end{equation*}
Since $\|\tilde u_n\|_{L^\infty(0,T; L^2(\Omega))}\leq \|M\|_{L^\infty(0,T)}$, there exist a subsequence of $\{\tilde u_n\}$, still denoted by itself, and $ u^*\in L^\infty(0,T; L^2(\Omega))$ such that 
\begin{equation}\label{1.15.1}
\tilde u_n\rightarrow u^* \mbox{ weakly}^* \mbox{ in } L^\infty(0,T; L^2(\Omega)).
\end{equation}
Take $t_0\in (0,\tau(\e))$ to be the Lebesgue point of $u^*$, and $\lambda\in(0,\frac{\tau(\e)-t_0}{2})$. Then there exists $N_0\in\mathbb{N}$ such that $t_0+\lambda\leq \tau_n$ for all $n\geq N_0$. For any $\zeta\in L^2(\Omega)$, set
\begin{equation*}
v(x,t)=\chi_{(t_0-\lambda, t_0+\lambda)}(t)\zeta(x).
\end{equation*}
Then $\|v\|_{L^1(0,T; L^2(\Omega))}=2\lambda\|\zeta\|<\infty$, and 
\begin{equation*}
\begin{split}
\int_{t_0-\lambda}^{t_0+\lambda} \langle  u^*(t), \zeta\rangle dt
&= \int_0^T\langle  u^*(t), \chi_{(t_0-\lambda, t_0+\lambda)}\zeta\rangle dt\\
&=\int_0^T\langle  u^*(t), v(t)\rangle dt\\
&=\lim_{n\rightarrow\infty}\int_0^T \langle \tilde u_n(t), v(t)\rangle dt\\
&=0.
\end{split}
\end{equation*}
Hence 
\begin{equation*}
\langle u^*(t_0), \zeta\rangle=\lim_{\lambda\rightarrow0}\frac{1}{2\lambda}\int_{t_0-\lambda}^{t_0+\lambda} \langle  u^*(t), \zeta\rangle dt=0.
\end{equation*}
By the arbitrary of $\zeta\in L^2(\Omega)$, we get $\tilde u^*(t_0)=0$. Since the Lebesgue measure of the set of $u^*$'s Lebesgue points in $(0,\tau(\e))$ is equal to $\tau(\e)$, we have 
\begin{equation}\label{1.15.2}
 u^*|_{(0,\tau(\e))}=0.
\end{equation} 

\ms

On the other side, by \eqref{1.15.1}, the solution $y^*(\cdot; \chi_{(\tau(\e),T)}u^*, y_0)$ to 
\begin{equation*}
\begin{cases}
\partial_t y^*(x,t)-\Delta y^*(x,t)=\chi_\omega\chi_{(\tau(\e),T)}u^*(x,t), &\mbox{ in } \Omega\times (0,T),\\
y^*(x,t)=0, &\mbox{ on }\partial \Omega\times (0,T),\\
y^*(x,0)=y_0(x), &\mbox{ in } \Omega,
\end{cases}
\end{equation*}
satisfies
\begin{equation}\label{1.15.3}
\begin{split}
y_n\rightarrow y^* & \mbox{ weakly in } L^2(0,T; H_0^1(\Omega))\cap H^1(0,T; L^2(\Omega)),\\
& \mbox{ strongly in } C([\delta, T]; L^2(\Omega)) \mbox{ as } n\rightarrow \infty,
\end{split}
\end{equation}
for any $0<\delta<T$. Here $y_n\equiv y_n(\cdot; \chi_{(\tau_n, T)}u_n, y_0)$ is the solution to the system
\begin{equation*}
\begin{cases}
\partial_t y_n(x,t)-\Delta y_n(x,t)=\chi_\omega\chi_{(\tau_n, T)} \tilde u_n(x,t), &\mbox{ in }\Omega\times (0,T),\\
y_n(x,t)=0, &\mbox{ on } \partial\Omega\times (0,T),\\
y_n(x,0)=y_0(x), &\mbox{ in }\Omega.
\end{cases}
\end{equation*}
Since $y_n(T; \chi_{(\tau_n,T)}u_n, y_0)\in \bar B(0,\e)$, we get $y^*(T; \chi_{(\tau(\e), T)}u^*, y_0)\in \bar B(0,\e)$ by \eqref{1.15.3}, which implies the optimal time $\tau(\e)$ is attainable and the optimal control $u^*$ exits. 

\ms

We claim that
 $\|u^*(t)\|\leq M(t)$ for a.e. $t\in (\tau(\e), T)$.

Indeed, if there exists $\e_0>0$ and  $E_0\subset [\tau(\e),T]$ with $|E_0|>0$ such that  
\begin{equation*}
\|u^*(t)\|>M(t)+\e_0,\ \forall t\in E_0,
\end{equation*}
where $|E_0|$ represents the Lebesgue measure of $E_0$. 
Define 
\begin{equation*}
\zeta(x,t)=\chi_{_{E_0}}\frac{u^*(t)}{\|u^*(t)\|}.
\end{equation*}
It is obviously that $\zeta(x,t)$ is well-defined since $\|u^*(\cdot)\|>M(\cdot)+\e_0>0$ in $E_0$, and 
\begin{equation*}
\|\zeta\|_{L^1(\tau(\e), T; L^2(\Omega))}=\int_{\tau(\e)}^T\|\zeta(t)\| dt=\int_{\tau(\e)}^T\left\|\chi_{_{E_0}}\frac{u^*(t)}{\|u^*(t)\|}\right\|dt=|E_0|<\infty,
\end{equation*}
i.e., $\zeta\in L^1(\tau(\e),T; L^2(\Omega))$. 

Since $\tilde u_n\rightarrow \tilde u^*$ weakly$^*$ in $L^\infty(0, T; L^2(\Omega))$, we get $u_n\rightarrow u^*$ weakly$^*$ in $L^\infty(\tau(\e),T; L^2(\Omega))$. For any $\epsilon\in(0,\frac{|E_0|\e_0}{2})$, there exists $N_0>0$, such that, for any $n\geq N_0$,   
\begin{equation}\label{1.15.5}
\left|\int_{\tau(\e)}^T \langle u_n-u^*, \zeta\rangle dt\right|<\epsilon.
\end{equation}
Noting 
\begin{equation*}
\|\zeta(t)\|=\left\|\chi_{_{E_0}}\frac{u^*(t)}{\|u^*(t)\|}\right\|=\chi_{_{E_0}},
\end{equation*}
we obtain
\begin{equation*}
\begin{split}
\left|\int_{\tau(\e)}^T\langle u_n-u^*, \zeta\rangle dt\right|
 =& \left|\int_{\tau(\e)}^T\langle u_n,\zeta\rangle dt-\int_{\tau(\e)}^T\langle u^*,\zeta\rangle dt\right|\\
 \geq& \left|\int_{\tau(\e)}^T\langle u^*,\zeta\rangle dt\right|-\left|\int_{\tau(\e)}^T\langle u_n,\zeta\rangle dt\right|\\
 \geq& \int_{\tau(\e)}^T\left\langle u^*, \chi_{E_0}\frac{u^*(t)}{\|u^*(t)\|}\right\rangle dt-\int_{\tau(\e)}^T\|u_n(t)\|\|\zeta(t)\|dt\\
 =& \int_{E_0}\|u^*(t)\|dt-\int_{E_0}\|u_n(t)\|dt\\
 \geq& \int_{E_0}(M(t)+\e_0)dt-\int_{E_0}M(t)dt\\
 =& |E_0|\e_0.
\end{split}
\end{equation*}
This contradicts \eqref{1.15.5}.
That proves our claim, and completes the proof.
\end{proof}

Now we can prove Theorem \ref{1.14.2}.

\begin{proof}[Proof of Theorem \ref{1.14.2}]
The proof is long, we separate it to two steps.

\ms

{\bf Step 1}. By Lemma \ref{1.24.8}, one knows that $\mathcal{R}(y_0,\tau(\e),T)\cap \bar B(0,\e)\neq \emptyset$. We now show that $\mathcal{R}(y_0,\tau(\e),T)\cap \bar B(0,\e)$ has only one point. 

\ms

Otherwise, there exists at least two different  $u_1^*, u_2^*\in \mathcal{U}_{M}$ such that 
\begin{equation*}
y_1\equiv y(T; \chi_{(\tau(\e),T)}u_1^*, y_0), y_2\equiv y(T; \chi_{(\tau(\e),T)}u_2^*, y_0)\in \bar B(0,\e)
\end{equation*}
and 
\begin{equation*}
y_1\neq y_2.
\end{equation*}
Note that $\ds \hat u=\frac{u_1^*}{2}+\frac{u_2^*}{2}\in \mathcal{U}_{M}$ and $\ds\hat y(\cdot)=\frac{1}{2}y_1(\cdot; \chi_{(\tau(\e),T)}u_1^*, y_0)+\frac{1}{2}y(\cdot; \chi_{(\tau(\e),T)}u_2^*, y_0)$ is the solution to the system
\begin{equation*}
\begin{cases}
\partial_t \hat y(x,t)-\Delta \hat y(x,t)=\chi_\omega\chi_{(\tau(\e),T)}\hat u(x,t), &\mbox{ in } \Omega\times (0,T),\\
\hat y(x,t)=0, &\mbox{ on }\partial \Omega\times (0,T),\\
\hat y(x,0)=y_0(x), &\mbox{ in } \Omega.
\end{cases}
\end{equation*}
One can easily get $\ds\hat y(T)=\frac{y_1}{2}+\frac{y_2}{2}\in \bar B(0,\e)$. Since $\bar B(0,\e)\subset L^2(\Omega)$ is strictly convex, we get $\hat y(T)$ is an inner point of $\bar B(0,\e)$. Hence there exists $\gamma>0$ such that $B(\hat y(T), \gamma)\subset B(0,\e)$. Let $\tilde y$ be the solution to  the following system
\begin{equation*}
\begin{cases}
\partial_t \tilde y(x,t)-\Delta \tilde y(x,t)=\chi_\omega\chi_{(\tau(\e)+\xi,T)}\hat u(x,t), &\mbox{ in } \Omega\times (0,T),\\
\tilde y(x,t)=0, &\mbox{ on }\partial \Omega\times (0,T),\\
\tilde y(x,0)=y_0(x), &\mbox{ in } \Omega.
\end{cases}
\end{equation*}
Then 
\begin{equation*}
h_\xi\equiv \hat y(T)-\tilde y(T)=\int_{\tau(\e)}^{\tau(\e)+\xi} e^{\Delta (T-\sigma)}\chi_\omega \hat u(\sigma)d \sigma.
\end{equation*}
Choosing $\xi>0$ small enough such that $\|h_\xi\|<\gamma$, one has $\tilde y(T)\in \bar B(0,\e)$. That implies $\tau(\e)\geq \tau(\e)+\xi$, which is impossible. That completes the proof of Step 1. 

\ms

{\bf Step 2}. The  optimal time control $u^*$ has the time-varying bang-bang property.

\ms

Since $\mathcal{R}(y_0,\tau(\e), T)\cap \bar B(0,\e)$ has only one point (denoted by $y^*=y(T; u^*, y_0)$),  and $\mathcal{R}(y_0,\tau(\e),T)$ and $\bar B(0,\e)$ are two convex sets, 
by hyperplane separation theorem, there exists $\eta^*\in L^2(\Omega)$ such that 
\begin{equation}\label{7.9.1}
\sup_{y\in \mathcal{R}(y_0,\tau(\e),T)}\langle y, \eta^*\rangle\leq \inf_{z\in \bar B(0,\e)} \langle z, \eta^*\rangle\leq \langle y^*, \eta^*\rangle.
\end{equation}
Notice that the element in $\mathcal{R}(y_0,\tau(\e),T)$ can be written by
\begin{equation*}
y(T; \chi_{(\tau(\e),T)} u, y_0)= e^{\Delta T}y_0+\int_0^{T} e^{\Delta(T-\sigma)} \chi_\omega \chi_{(\tau(\e),T)}u(\sigma)\df \sigma.
\end{equation*}
Then by \eqref{7.9.1}, one can get
\begin{equation*} 
\sup_{\bar u\in\mathcal{U}_1}\int_{\tau(\e)}^{T}\left\langle e^{\Delta(T-\sigma)}\chi_\omega M(\sigma)\bar u(\sigma), \eta^*\right\rangle\df \sigma\leq  \int_{\tau(\e)}^{T}\left\langle e^{\Delta(T-\sigma)}\chi_\omega M(\sigma)\bar u^*(\sigma), \eta^*\right\rangle\df \sigma,
\end{equation*}
i.e.,
\begin{equation}\label{wanga1} 
\sup_{\bar u\in\mathcal{U}_1}\int_{\tau(\e)}^{T}\left\langle \bar u(\sigma), e^{\Delta(T-\sigma)}\chi_\omega M(\sigma)\eta^*\right\rangle\df \sigma\leq  \int_{\tau(\e)}^{T}\left\langle \bar u^*(\sigma), e^{\Delta(T-\sigma)}\chi_\omega M(\sigma)\eta^*\right\rangle\df \sigma.
\end{equation}
Here 
\begin{equation*}
\bar u^*\in\mathcal{U}_1\equiv \left\{\bar u\in L^\infty(\tau(\e),T; L^2(\Omega))\bigm| \|\bar u(t)\|_{L^2(\Omega)}\leq 1 \mbox{ for a.e. } t\in [\tau(\e), T]\right\},
\end{equation*}
and 
\begin{equation}\label{7.9.2}
u^*(t)=M(t)\bar u^*(t) \mbox{ for all } t\in [\tau(\e),T].
\end{equation}

Let $E_0$ be the Lebesgue points of $\bar u^*(\cdot)$ in $[\tau(\e), T]$. For given $t_0\in E_0$, choosing 
\begin{equation*}
\bar u(t)=
\begin{cases}
\bar u^*(t), &\mbox{ for } t\in (\tau(\e), T)\setminus (t_0-\lambda, t_0+\lambda),\\
\zeta, &\mbox{ for } t\in (t_0-\lambda, t_0+\lambda)\subset (\tau(\e), T),
\end{cases}
\end{equation*}
where $\zeta\in L^2(\O)$ with $\|\zeta\|\leq 1$, and $\l\in (0,\min\{t_0-\tau(\e),T-t_0\})$. Setting $A=(\tau(\e), T)\setminus(t_0-\lambda, t_0+\lambda)$, by \eqref{wanga1} we have 
\begin{equation*}
\begin{split}
 &\int_A\langle \bar u^*(t), e^{\Delta (T-\sigma)}M(\sigma)\eta^*\rangle d\sigma+\int_{t_0-\lambda}^{t_0+\lambda}\langle \zeta, e^{\Delta (T-\sigma)}\chi_\omega M(\sigma)\eta^*\rangle d\sigma\\
 &\leq \int_A\langle \bar u^*(t), e^{\Delta (T-\sigma)}M(\sigma)\eta^*\rangle d\sigma+\int_{t_0-\lambda}^{t_0+\lambda}\langle \bar u^*(\sigma), e^{\Delta(T-\sigma)}\chi_\omega M(\sigma)\eta^*\rangle d\sigma,
\end{split}
\end{equation*}
i.e.,
\begin{equation*}
\int_{t_0-\lambda}^{t_0+\lambda}\langle \zeta, e^{\Delta (T-\sigma)}\chi_\omega M(\sigma)\eta^*\rangle d\sigma\leq \int_{t_0-\lambda}^{t_0+\lambda}\langle \bar u^*(\sigma), e^{\Delta(T-\sigma)}\chi_\omega M(\sigma)\eta^*\rangle d\sigma.
\end{equation*}
Hence
\begin{equation*}
\begin{split}
 &\langle \zeta, e^{\Delta (T-t_0)}\chi_\omega M(t_0)\eta^*\rangle\\
 &=\lim_{\lambda\rightarrow0}\frac{1}{2\lambda}\int_{t_0-\lambda}^{t_0+\lambda}\langle \zeta, e^{\Delta (T-\sigma)}\chi_\omega M(\sigma)\eta^*\rangle d\sigma\\
 &\leq \lim_{\lambda\rightarrow0}\frac{1}{2\lambda}\int_{t_0-\lambda}^{t_0+\lambda}\langle \bar u^*(\sigma), e^{\Delta(T-\sigma)}\chi_\omega M(\sigma)\eta^*\rangle d\sigma\\
 &=\langle \bar u^*(t_0), e^{\Delta(T-t_0)}\chi_\omega M(t_0)\eta^*\rangle.
\end{split}
\end{equation*}
By the arbitrary of $\zeta$, we get 
\begin{equation*}
\sup_{\|\zeta\|\leq 1} \left\langle \zeta, e^{\Delta(T-t_0)}\chi_\omega M(t_0)\eta^*\right\rangle\leq \left\langle \bar u^*(t_0), e^{\Delta(T-t_0)}\chi_\omega M(t_0)\eta^*\right\rangle,
\end{equation*}
i.e., 
\begin{equation*}
\begin{split}
\left\|e^{\Delta(T-t_0)}\chi_\omega M(t_0)\eta^*\right\|_{L^2(\Omega)}
 &=\sup_{\|\zeta\|\leq 1} \left\langle \zeta, e^{\Delta(T-t_0)}\chi_\omega M(t_0) \eta^*\right\rangle\\
 &\leq \left\langle \bar u^*(t_0), e^{\Delta(T-t_0)}\chi_\omega M(t_0) \eta^*\right\rangle\\
 &\leq \|\bar u^*(t_0)\|_{L^2(\Omega)}\left\|e^{\Delta(T-t_0)}\chi_\omega M(t_0) \eta^*\right\|_{L^2(\Omega)}.
\end{split}
\end{equation*}
This implies that 
\begin{equation}\label{wang11}
\|\bar u^*(t_0)\|_{L^2(\Omega)}\geq 1.
\end{equation}
By $\bar u^*\in \mathcal{U}_1$ we get $\|\bar u^*(t_0)\|=1$. \eqref{wang11}, together with \eqref{7.9.2} and $|E_0|=\hat\tau(\e)$, yields 
\begin{equation*}
\|u^*(t)\|_{L^2(\Omega)}=M(t) \mbox{ for a.e. } t\in [0,\hat\tau(\e)].
\end{equation*}
From above, we get the time optimal control $u^*$ satisfies  the time-varying bang-bang property. 
That completes the proof.
\end{proof}

\section{Proof of Theorem \ref{1.14.4}}

In order to show Theorem \ref{1.14.4}, we need a lemma in the following:

\begin{lemma}\label{1.24.6}
Let $\e(\tau)$ be defined as \eqref{1.24.111}. Then there exists $u^*\in\mathcal{U}_{M}$ such that 
\begin{equation*}
y(T; \chi_{( \tau, T)}u^*, y_0)\in \bar B(0,\e(\tau)). 
\end{equation*}
\end{lemma}

\begin{proof}
Let $\{u_n\}_{n=1}^\infty\subset\mathcal{U}_{M}$ be the minimal sequence of \eqref{1.24.111}, i.e., 
\begin{equation*}
\|y(T; \chi_{(\tau, T)}u_n, y_0)\|\rightarrow \e(\tau),\ \mbox{as}\  n\rightarrow \infty.
\end{equation*}
Here $y(\cdot; \chi_{(\tau, T)}u_n, y_0)$ is the solution to the system
\begin{equation*}
\begin{cases}
\partial_t y(x,t)-\Delta y(x,t)=\chi_\omega\chi_{(\tau,T)}u_n(x,t), &\mbox{ in }\Omega\times (0,T),\\
y(x,t)=0, &\mbox{ on }\partial\Omega\times (0,T),\\
y(x,0)=y_0(x), &\mbox{ in } \Omega.
\end{cases}
\end{equation*}
Since 
\begin{equation*}
\|u_n(t)\|\leq M(t)\leq\|M\|_{L^\infty(0,T)} \mbox{ for a.e. } t\in (\tau, T),
\end{equation*}
there exist a subsequence of $\{u_n\}_{n=1}^\infty$, still denoted by itself, and $u^*\in L^\infty(\tau, T; L^2(\Omega))$ such that 
\begin{equation*}
u_n\rightarrow u^* \mbox{ weakly}^* \mbox{ in } L^\infty(\tau, T; L^2(\Omega)).
\end{equation*}
Therefore, we have 
\begin{equation*}
\begin{split}
y(\cdot; \chi_{(\tau, T)}u_n, y_0)\rightarrow y(\cdot; \chi_{(\tau, T)}u^*, y_0) & \mbox{ weakly in } L^2(0,T; H_0^1(\Omega))\cap H^1(0,T; L^2(\Omega)),\\
& \mbox{ strongly in } C([\delta, T]; L^2(\Omega)) \mbox{ as } n\rightarrow \infty,
\end{split}
\end{equation*}
for any $\delta\in(0,T]$. Here $y(\cdot; \chi_{(\tau, T)}u^*, y_0)$ is the solution to the system
\begin{equation*}
\begin{cases}
\partial_t y(x,t)-\Delta y(x,t)=\chi_\omega\chi_{(\tau,T)}u^*(x,t), &\mbox{ in }\Omega\times (0,T),\\
y(x,t)=0, &\mbox{ on }\partial\Omega\times (0,T),\\
y(x,0)=y_0(x), &\mbox{ in } \Omega.
\end{cases}
\end{equation*}
Hence, 
\begin{equation*}
\|y(T; \chi_{(\tau, T)}u^*,y_0)\|=\e(\tau).
\end{equation*}

\ms

Now, we shall show that $u^*\in \mathcal{U}_{M}$. 

\ms

By contradiction. We assume there exist $\epsilon_0>0$ and $E_0\subset (\tau, T)$ with $|E_0|>0$ such that 
\begin{equation*}
u^*(t)\geq M(t)+\epsilon_0 \mbox{ for all } t\in E_0.
\end{equation*}
 Then
\begin{equation}\label{wang1}
\int_{E_0}\|u^*(t)\|dt \geq \int_{E_0} \Big(M(t)+\epsilon_0\Big)dt=\int_{E_0}M(t)dt+\epsilon_0|E_0|.
\end{equation}
Taking  $\zeta(x,t)=\chi_{_{E_0}}\frac{u^*(t)}{\|u^*(t)\|}$, by the assumption of $u^*$, we get 
\begin{equation*}
\int_0^T \| \zeta(x,t)\|dt=\int_0^T\left\|\chi_{_{E_0}}\frac{u^*(t)}{\|u^*(t)\|}\right\|dt =|E_0|,
\end{equation*}
i.e., $\zeta\in L^1(0, T; L^2(\Omega))$.  
Hence, by $u_n\rightarrow u^*$ weakly$^*$ in $L^\infty(\tau,T; L^2(\Omega))$, one has
\begin{equation*}
\int_\tau^T\left\langle u_n(t), \chi_{_{E_0}}\frac{u^*(t)}{\|u^*(t)\|}\right\rangle dt=\int_\tau^T\langle u_n(t), \zeta\rangle dt\rightarrow  \int_\tau^T\langle u^*(t), \zeta(t)\rangle dt=\int_{E_0} \|u^*(t)\|dt, 
\end{equation*}
as $n\rightarrow \infty$.
On the other hand, since $u_n\rightarrow u^*$ weakly$^*$ in $L^\infty(\tau,T; L^2(\Omega))$ and $E_0\subset (\tau,T)$, we get $u_n\rightarrow u^*$ weakly$^*$ in $L^\infty(E_0; L^2(\Omega))$. Therefore, by $\|\zeta(t)\|=\left\| \chi_{_{E_0}}\frac{u^*(t)}{\|u^*(t)\|}\right\|=1$ and $\|u_n(t)\|\leq M(t)$ for a.e. $t\in (\tau, T)$, one gets
\begin{equation*}
\begin{split}
\int_{E_0} \|u^*(t)\|dt 
 &=\lim_{n\rightarrow \infty} \int_{E_0} \langle u_n(t), \zeta(t)\rangle dt\\
 &\leq \lim_{n\rightarrow \infty} \int_{E_0} \|u_n(t)\|\|\zeta(t)\|dt\\
 &=\lim_{n\rightarrow \infty} \int_{E_0} \|u_n(t)\| dt\\
 &\leq \int_{E_0}M(t)dt,
\end{split}
\end{equation*}
which is contradict with \eqref{wang1}. That completes the proof.
\end{proof}

Now, we are in the position to prove Theorem \ref{1.14.4}.

\begin{proof}[Proof of Theorem \ref{1.14.4}]
We carry out the proof by three steps as follows.

\ms

{\bf Step 1}. We shall show that $\tau\mapsto \e(\tau)$ is strictly monotonically increasing.

Let $0\leq \tau_1<\tau_2<T$. For $\tau_2$, by Theorem \ref{1.14.2}, 
there exists a unique $u_2\in\mathcal{U}_{M}$ such that 
\begin{equation*}
\|y(T; \chi_{(\tau_2,T)}u_2, y_0)\|=\e(\tau_2).
\end{equation*}
Now, take 
\begin{equation*}
u_1(t)=
\begin{cases}
u_2(t), &\mbox{ if } t\in (\tau_2, T),\\
0, &\mbox{ if } t\in [0,\tau_2].
\end{cases}
\end{equation*}
Then 
\begin{equation*}
\|y(T; \chi_{(\tau_1, T)}u_1, y_0\|=\|y(T; \chi_{(\tau_2, T)}u_2, y_0)\|\in \bar B(0,\e(\tau_2)).
\end{equation*}
By the definition of $\e(\tau)$ we get  
\begin{equation*}
\e(\tau_1)\leq\e(\tau_2).
\end{equation*}
In other words, $\tau\mapsto \e(\tau)$ is a monotonically increasing function.

\ms

Now, we show that $\tau\mapsto \e(\tau)$ is  strictly monotonically increasing.
If not, suppose that $\e(\tau_1)=\e(\tau_2)$. Then, there exist $u_1,u_2\in\mathcal{U}_{M}$, such that 
\begin{equation*}
\|y(T; \chi_{(\tau_1,T)}u_1, y_0)\|=\e(\tau_1)=\e(\tau_2)=\|y(T; \chi_{(\tau_2, T)}u_2, y_0)\|.
\end{equation*}
Taking
\begin{equation*}
\hat u_2(t)=
\begin{cases}
u_2(t), &\mbox{if } t\in (\tau_2,T),\\
0, &\mbox{if } t\in [0,\tau_2].
\end{cases}
\end{equation*}
one can easily check that
\begin{equation*}
\|y(T; \chi_{(\tau_1,T)}u_1,y_0)\|=\e(T_1)=\|y(T; \chi_{(\tau_1,T)}\hat u_2, y_0)\|.
\end{equation*}
By Theorem \ref{1.14.2}, the optimal control is unique. Hence, we get 
\begin{equation*}
u_1(t)=\hat u_2(t) \mbox{ for a.e. } t\in (0,T).
\end{equation*}
By the bang-bang property of minimal control, we have 
\begin{equation*}
\|u_1(t)\|=\|\hat u_2(t)\|=M(t) \mbox{ for a.e. } t\in (0,T).
\end{equation*}
That is impossible, since $\hat u_2(t)=0$ for a.e. $t\in (\tau_1,\tau_2)$ and $M(t)>0$ for a.e. $t\in (0,T)$. Therefore, $\tau\mapsto\e(\tau)$ is a strictly monotonically increasing function. 

\ms

{\bf Step 2}. We shall show that $\tau\mapsto \e(\tau)$ is continuous. 

For $\tau, \hat\tau\in [0,T)$ with $\tau<\hat \tau$, and for each $u\in \mathcal{U}_{M}$, the solutions to \eqref{6.29.1} are 
\begin{equation*}
y(T; \chi_{(\tau, T)}u, y_0)=e^{\Delta T}y_0+\int_0^T e^{\Delta (T-\sigma)}\chi_\omega\chi_{(\tau,T)} u(\sigma) d \sigma
\end{equation*}
and
\begin{equation*}
y(T; \chi_{(\hat\tau, T)}u, y_0)=e^{\Delta T}y_0+\int_0^T e^{\Delta (T-\sigma)}\chi_\omega\chi_{(\hat\tau,T)} u(\sigma) d \sigma,
\end{equation*}
respectively.
Then 
\begin{equation}\label{1.24.3}
\begin{split}
 &\|y(T; \chi_{(\tau,T)}u,y_0)- y(T; \chi_{(\hat\tau, T)}u, y_0)\|\\
 &= \left\|\int_0^T e^{\Delta (T-\sigma)}\chi_\omega\chi_{(\tau,T)} u(\sigma) d \sigma-\int_0^T e^{\Delta (T-\sigma)}\chi_\omega\chi_{(\hat\tau,T)} u(\sigma) d \sigma\right\|\\
 &=\left\| \int_\tau^{\hat \tau} e^{\Delta (T-\sigma)}\chi_\omega u(\sigma)d \sigma\right\|\\
 &\leq C|\hat \tau-\tau|,
\end{split}
\end{equation}
where $C$ is a constant independent of $\tau,\hat\tau$. Since $u\in \mathcal{U}_{M}$ (i.e., $\|u(t)\|\leq M(t) \leq M_T$ for a.e. $t\in [0,T]$), by \eqref{1.24.3},  we get $\mbox{dist}\big(\mathcal{R}(y_0,\tau,T),\mathcal{R}(y_0,\hat\tau,T)\big)<C|\tau-\hat\tau|$. Here $\mbox{dist}$ is the distance of two reachable sets $\mathcal{R}(y_0,\tau,T)$ and $\mathcal{R}(y_0,\hat\tau, T)$. Hence $\tau\mapsto \e(\tau)$ is a continuous function.

\ms

{\bf Step 3}. We shall prove \eqref{1.14.5}.

\ms

(1) We show that $\e=\e(\tau(\e))$ for $\e\in [\e(0), \e_T)$. 

Let $\e\in [\e(0),\e_T)$. By Step 1 in the proof of Theorem \ref{1.14.2}, there exist $\tau(\e)\in [0,T)$ and $u\in \mathcal{U}_{M}$ such that 
\begin{equation}\label{1.24.5}
\|y(T; \chi_{(\tau(\e), T)}u, y_0)\|=\e.
\end{equation}
For such $\tau(\e)\in [0,T)$, denote $\tau^*=\tau(\e)$. We consider the following problem
\begin{equation*}
\e(\tau^*)=\inf_{u\in\mathcal{U}_{M}} \|y(T; \chi_{(\tau^*, T)}u, y_0)\|.
\end{equation*}
By \eqref{1.24.5}, we can obtain
\begin{equation}\label{wang2}
\e(\tau^*)\leq \e. 
\end{equation}
By lemma \ref{1.24.6}, there exists a control $u^*\in\mathcal{U}_{M}$ such that 
\begin{equation*}
\|y(T; \chi_{(\tau^*, T)}u^*, y_0)\|=\e(\tau^*).
\end{equation*}
Now, taking 
\begin{equation*}
\tilde u(t)=
\begin{cases}
u^*(t), &\mbox{ if } t\in (\tau^*, T),\\
0, &\mbox{ if } t\in [0,\tau^*],
\end{cases}
\end{equation*}
 we have 
\begin{equation*}
\|y(T; \chi_{(\tau^*, T)}\tilde u, y_0)\|=\e(\tau^*).
\end{equation*}
By the definition of $\tau(\e)$ and \eqref{1.24.5} we get 
\begin{equation*}\label{wang3}
\e\leq \e(\tau^*).
\end{equation*}
which, together with \eqref{wang2}, yields
\begin{equation*}
\e=\e(\tau(\e)).
\end{equation*}

(2) We show that $\tau=\tau(\e(\tau))$ for $\tau\in [0,T)$. 

Let $\tau\in [0,T)$. By Lemma \ref{1.24.6}, there exists $u\in \mathcal{U}_{M}$ such that 
\begin{equation}\label{1.25.2}
y(T; \chi_{(\tau, T)}u, y_0)\in \bar B(0,\e(\tau)).
\end{equation}
For the given $\e(\tau)\in [\e(0), \e_T)$, denote $\e^*=\e(\tau)$. Consider the following problem 
\begin{equation*}
\tau(\e^*)=\sup \{\tilde \tau\bigm| y(T;\chi_{(\tilde \tau, T)}u, y_0)\in \bar B(0,\e^*),\ u\in\mathcal{U}_{M}\}.
\end{equation*}
Then we have 
\begin{equation}\label{wang4}
\tau\leq \tau(\e^*).
\end{equation}
By Lemma \ref{1.24.8}, there exists a control $u^*\in\mathcal{U}_{M}$ such that 
\begin{equation*}
y(T; \chi_{(\tau(\e^*), T)}u^*, y_0)\in \bar B(0,\e^*).
\end{equation*}
Then by the definition of $\e(\tau)$ and \eqref{1.25.2} we get
\begin{equation*}
\tau(\e^*)\leq \tau,
\end{equation*}
which, together with \eqref{wang4}, yields $\tau=\tau(\e(\tau))$ for $\tau\in [0,T)$. 
That completes the proof.
\end{proof}

\section*{Acknowledgement} 
The third author gratefully acknowledges Dr. Yubiao Zhang, Wuhan University, for his helpful 
discussion during this work.


%

\begin{thebibliography}{10}

\bibitem{Apraiz-Escauriaza-Wang-Zhang14}
{\sc J.~Apraiz, L.~Escauriaza, G.~Wang, and C.~Zhang}, {\em Observability
  inequalities and measurable sets}, J. Eur. Math. Soc., 16 (2014),
  pp.~2433--2475.
  
\bibitem{FPZ} 
{\sc C. Faber, J.-P. Puel, and E. Zuazua}, {\em Approximate controllability of the semilinear heat equation}, Proc. Roy. Soc. Edinburgh, 125A (1995), pp.~31--61.

\bibitem{Fattorini99}
{\sc H.~O. Fattorini}, {\em Infinite-dimensional Optimization and Control
  Theory}, vol.~62 of Encyclopedia of Mathematics and its Applications,
  Cambridge University Press, Cambridge, 1999.

\bibitem{Fattorini11}
{\sc H.~O. Fattorini}, {\em Time and norm
  optimal controls: a survey of recent results and open problems}, Acta Math.
  Sci. Ser. B Engl. Ed., 31 (2011), pp.~2203--2218.
  
\bibitem{FCZ}
{\sc E. Fernandez-Cara and E. Zuazua}, {\em Null and approximate controllability for weakly blowing-up semilinear heat equations}, Ann. Inst. H. Poincar\'e, Anal. Non Lineaire, 17 (2000), pp.~583--616.

\bibitem{GL}
{\sc F. Gozzi and P. Loreti}, {\em Regularity of the minimum time function and minimum energy problems: the linear case}, SIAM J. Control Optim., 37 (1999), pp.~1195--1221.

\bibitem{Guo-Xu-Yang16}
{\sc B.-Z. Guo, Y.~Xu, and D.-H. Yang}, {\em Optimal actuator location of
  minimum norm controls for heat equation with general controlled domain}, J.
  Differential Equations 261 (2016),  pp.~3588--3614.

\bibitem{Guo-Yang15}
{\sc B.-Z. Guo and D.-H. Yang}, {\em Optimal actuator location for time and
  norm optimal control of null controllable heat equation}, Math. Control
  Signals Systems, 27 (2015), pp.~23--48.
  
\bibitem{KW}
{\sc K. Kunish and L. Wang}, {\em Time optimal control of the heat equation with pointwise control constraints}, ESAIM control Optim. Calc. Var., 19 (2013), pp.~460--485.

\bibitem{LY}
{\sc X. Li and J. Yong}, {\em Optimal Control Theory for Infinite Dimensional Systems}. Birkh\"auser Boston, Boston, MA, 1995.

\bibitem{Lv10}
{\sc Q.~L{{\"u}}}, {\em Bang-bang principle of time optimal controls and null controllability of fractional order parabolic equations}, Acta Math. Sin., English Series, 26 (2010), pp.~2377--2386.

  
\bibitem{MS}
{\sc V. Mizel and T. Seidman}, {\em An abstract bang-bang principle and time optimal boundary control of the heat equation}, SIAM J. Control Optim., 35 (1997), pp. 1204--1216.

\bibitem{S}
{\sc E. J. P. Georg Schmidt}, {\em The ``Bang-Bang'' principle for the time-optimal problem in boundary control of the heat equation}, SIAM J. Control Optim., 18 (1980), pp.101--107. 

\bibitem{Wang08}
{\sc G. Wang}, {\em {$L^\infty$}-null controllability for the heat equation and
  its consequences for the time optimal control problem}, SIAM J. Control
  Optim., 47 (2008), pp.~1701--1720.

\bibitem{Wang-Xu13}
{\sc G. Wang and Y. Xu}, {\em Equivalence of three different kinds of optimal control problems for heat equations and its applications}, SIAM J. Control
  Optim., 51 (2013), pp.~848--880.
  
\bibitem{WZhang}
{\sc G. Wang and Y. Zhang}, {\em Decompositions and bang-bang properties}, Math. Control Relat. Fields, 7 (2017), pp.~73--170.
  
\bibitem{WZ}
{G. Wang and E. Zuazua}, {\em On the equivalence of minimal time and minimal norm controls for internal controlled heat equations}, SIAM J. Control Optim., 50 (2012), pp. 2938--2958.

\bibitem{Wang-Yang-Yong-Yu16}
{\sc Y. Wang, D.-H. Yang, J. Yong, and Z. Yu}, {\em Exact controllability of
  linear stochastic differential equations and related problems},
  arXiv:1603.07789  (2016).

\bibitem{Wang-Zhang15}
{\sc Y. Wang and C. Zhang}, {\em The norm optimal control problem for
  stochastic linear control systems}, ESAIM Control Optim. Calc. Var., 21
  (2015), pp.~399--413.

\bibitem{Yang-Zhong16}
{\sc D.-H. Yang and J. Zhong}, {\em Observability inequality of backward
  stochastic heat equations for measurable sets and its
  applications}, SIAM J. Control Optim., 54 (2016), pp.~1157--1175.

\bibitem{Yong-Zhou99}
{\sc J. Yong and X. Y. Zhou}, {\em Stochastic Controls: Hamiltonian Systems and HJB Equations}, vol.~43 of
  Applications of Mathematics (New York), Springer-Verlag, New York, 1999.

\bibitem{Zuazua06}
{\sc E. Zuazua}, {\em Controllability of Partial Differential Equations},
  manuscript, 2006.

\end{thebibliography}

\end{document}